\documentclass[11pt,twoside]{article}

\usepackage{amssymb}
\usepackage{amsmath}
\usepackage{amsthm}

\allowdisplaybreaks

\def\rr{{\mathbb R}}
\def\rn{{{\rr}^n}}

\def\zz{{\mathbb Z}}
\def\nn{{\mathbb N}}

\def\cc{{\mathbb C}}

\def\cg{{\mathcal G}}

\def\cm{{\mathcal M}}

\def\cmn{{\Phi^\ast}}
\def\cmr{{\Phi^+}}
\def\cmg{{G^{(\ez,\,\bz,\,\gz)}}}

\def\fz{\infty}
\def\az{\alpha}
\def\supp{{\mathop\mathrm{\,supp\,}}}

\def\lz{\lambda}
\def\dz{\delta}

\def\ez{\epsilon}

\def\bz{\beta}

\def\gz{{\gamma}}

\def\bgz{{\Gamma}}

\def\vz{\varphi}

\def\pa{\partial}
\def\wz{\widetilde}

\def\ls{\lesssim}
\def\gs{\gtrsim}

\def\tbz{{\triangle_\lz}}
\def\dmz{{dm_\lz}}
\def\zlz{{Z^{[\lz]}}}
\def\tlz{{\tau^{[\lz]}_x}}
\def\jz{{J_\nu}}
\def\riz{{R_{\Delta_\lz}}}
\def\triz{{\tilde R_{\Delta_\lz}}}
\def\slz{{\sharp_\lz}}
\def\qlz{{Q^{[\lz]}_t}}
\def\plz{{P^{[\lz]}_t}}
\def\ati{\mathop\mathrm{AOTI}}

\def\ocg{{\mathring{\cg}}}

\def\bmoz{{\rm BMO}((0, \fz),\, dm_\lz)}

\def\ltz{{L^2((0,\, \fz),\, dm_\lz)}}
\def\lrz{{L^r((0,\, \fz),\, dm_\lz)}}
\def\lrpz{{L^{r'}((0,\, \fz),\, dm_\lz)}}
\def\loz{{L^1((0,\, \fz),\, dm_\lz)}}
\def\lpz{{L^p((0,\, \fz),\, dm_\lz)}}
\def\linz{{L^\fz((0,\, \fz),\, dm_\lz)}}
\def\hpz{{H^p((0,\, \fz),\, dm_\lz)}}
\def\hoz{{H^1((0,\, \fz),\, dm_\lz)}}

\def\lp{{L^p(\rn)}}

\def\dint{\displaystyle\int}

\def\dlimsup{\displaystyle\limsup}
\def\dfrac{\displaystyle\frac}
\def\dsup{\displaystyle\sup}
\def\dlim{\displaystyle\lim}

\def\r{\right}
\def\lf{\left}

\newtheorem{thm}{Theorem}[section]
\newtheorem{lem}{Lemma}[section]

\newtheorem{rem}{Remark}[section]
\newtheorem{cor}{Corollary}[section]
\newtheorem{defn}{Definition}[section]

\numberwithin{equation}{section}

\begin{document}

\arraycolsep=1pt

\title{{\vspace{-5cm}\small\hfill\bf Anal. Appl. (Singap.) (to appear)}\\
\vspace{4.5cm}\Large\bf REAL-VARIABLE CHARACTERIZATIONS OF HARDY SPACES
ASSOCIATED WITH BESSEL OPERATORS}
\author{DACHUN YANG\\
\it \footnotesize School of Mathematical Sciences, Beijing Normal University,\\
\it\footnotesize Laboratory of Mathematics and Complex Systems, Ministry of Education, \\
\it\footnotesize Beijing 100875, People's Republic of China.\\
\vspace{0.6cm}\it\footnotesize dcyang@bnu.edu.cn\\
DONGYONG YANG\\
\it \footnotesize School of Mathematical Sciences, Xiamen University,\\
\it\footnotesize Xiamen 361005, People's Republic of China.\\
{\it\footnotesize dyyangbnu@yahoo.com.cn}}
\date{ }
\maketitle

\begin{center}
\begin{minipage}{13.5cm}\small
{\noindent Let $\lambda>0$, $p\in((2\lz+1)/(2\lz+2), 1]$,  and
$\triangle_\lambda\equiv-\frac{d^2}{dx^2}-\frac{2\lambda}{x}
\frac d{dx}$ be the Bessel operator. In this paper, the authors
establish the characterizations of atomic Hardy spaces
$H^p((0, \infty), dm_\lambda)$ associated with $\triangle_\lambda$
in terms of the radial maximal function, the nontangential maximal function,
the grand maximal function, the Littlewood-Paley $g$-function and
the Lusin-area function,
where $dm_\lambda(x)\equiv x^{2\lambda}\,dx$.
As an application, the authors further obtain the Riesz transform
characterization of these Hardy spaces.}

\medskip


\medskip

\end{minipage}
\end{center}

\section{\hspace{-0.6cm}.\hspace{0.4cm}Introduction\label{s1}}

\noindent It is well known that
the real-variable theory of Hardy spaces on the
$n$-dimensional Euclidean space $\rn$ plays an important
role in harmonic analysis and has been systematically developed;
see \cite{fs72,g08,s93,sw60}. The classical Hardy spaces on $\rn$
are essentially related to the Laplacian $\triangle\equiv
-\sum_{k=1}^n\frac{\pa^2}{\pa x_k^2}$.

Let $\lz\in(0, \fz)$ and $\tbz$ be the Bessel operator,
which is defined by setting, for all $C^2$-functions $f$ on $(0, \fz)$
and $x\in (0, \fz)$,
\begin{equation*}
\tbz f(x)\equiv-\frac{d^2}{dx^2}f(x)-\frac{2\lz}{x}\frac{d}{dx}f(x).
\end{equation*}
In 1965, Muckenhoupt and Stein \cite{ms} developed a theory parallel
to the classical case associated to $\triangle$ in the setting of
$\tbz$, in which the results on $\lpz$-boundedness of conjugate
functions and fractional integrals associated with $\tbz$ were
obtained, where $p\in[1, \fz)$ and $\dmz(x)\equiv x^{2\lz}\,dx$.
Since then, many problems based on the Bessel context were studied;
see, for example, \cite{ak,bcfr,bfbmt,bfs,bhnv,k78,s85,v08}. In
particular, Betancor et al. in \cite{bdt} established the
characterizations of the atomic Hardy space $H^1((0, \fz), \dmz)$
associated to $\tbz$ in terms of the Riesz transform and the radial
maximal function associated with the Hankel convolution of a class
$\zlz$ of functions, which includes the Poisson semigroup and the heat
semigroup as special cases.

Let $p\in((2\lz+1)/(2\lz+2), 1]$. The main target of this paper is,
via using the results from \cite{hmy06, hmy}, to establish the
characterizations of the atomic Hardy spaces $\hpz$ in terms of the
radial maximal function, the nontangential maximal function, the grand
maximal function, the Littlewood-Paley $g$-function and
the Lusin-area function. As an application,
we further obtain the Riesz transform
characterization of these Hardy spaces.

To state our main results, we first recall some necessary notions and
notation.
Throughout this paper, we assume that $\lz\in (0, \fz)$. Let
$\bgz$ and $\jz$ respectively denote the {\it Gamma function} and
the {\it Bessel function of the first kind of order $\nu$} with $\nu\in(-1/2, \fz)$.
For any $f$ and $g\in \loz$, their {\it Hankel convolution} is defined by
setting, for all $x\in (0,\fz)$,
\begin{equation}\label{1.1}
f\sharp_\lz g(x)\equiv\dint_0^\fz f(y)\tlz g(y)\dmz(y),
\end{equation}
where for $x\in (0, \fz)$, $\tlz g(y)$
denotes the {\it Hankel translation} of $g(y)$, that is,
\begin{equation}\label{1.2}
\tlz g(y)\equiv\frac{\bgz(\lz+1/2)}{\bgz(\lz)\sqrt{\pi}}
\dint_0^\pi g\lf(\sqrt{x^2+y^2-2xy\cos \theta}\r)(\sin\theta)^{2\lz-1}\,d\theta.
\end{equation}

In what follows, for any $x$, $r\in (0, \fz)$, let the {\it symbol}
$I(x, r)\equiv(x-r, x+r)\cap(0, \fz)$. It is easy to see that
\begin{equation}\label{1.3}
m_\lz(I(x, r))\sim\left\{
                    \begin{array}{ll}
                      x^{2\lz}r,\,\, & \hbox{$x>r$;} \\
                      r^{2\lz+1},\,\, & \hbox{$x\le r$.}
                    \end{array}
                  \right.
\end{equation}
This yields that
\begin{equation}\label{1.4}
m_\lz(I(x, r))\sim x^{2\lz}r+r^{2\lz+1},
\end{equation}
which further implies that
\begin{equation}\label{1.5}
2m_\lz(I(x, r))\ls m_\lz(I(x, 2r))\ls 2^{2\lz+1}m_\lz(I(x, r)).
\end{equation}
Thus, $((0, \fz), \rho, dm_\lz)$ is an RD-space introduced in \cite{hmy},
where $\rho(x,y)\equiv|x-y|$ for all $x,\,y\in(0, \fz)$.
We now recall the notion of approximations of the identity in the
context of RD-spaces, which was introduced in \cite{hmy} (see also \cite{hmy06}).

\begin{defn}\label{d1.1}\rm
Let $\ez_1\in(0, 1]$ and $\ez_2$, $\ez_3\in(0, \fz)$.
A sequence $\{S_k\}_{k\in\zz}$ of bounded linear integral operators on
$\ltz$ is called an {\it approximation of the identity
of order $(\ez_1, \ez_2, \ez_3)$}
(in short, $(\ez_1, \ez_2, \ez_3)-\ati$),
if there exists a positive constant $\wz C$ such that for all $k\in\zz$ and
$x,\,y\in (0, \fz)$, $S_k(x, y)$, the integral kernel of $S_k$ is a
 measurable function from $(0, \fz)\times
(0, \fz)$ into $\cc$ satisfying that
\begin{enumerate}

\item[(i)] for all $k\in\zz$ and $x,\,y\in (0, \fz)$,
\begin{equation*}
|S_k(x,\,y)|\le \wz C\frac 1{m_\lz(I(x, 2^{-k})) +m_\lz(I(y, 2^{-k}))+
m_\lz(I(x, |x-y|))}\frac {2^{-k\ez_2}}{(2^{-k}+|x-y|)^{\ez_2}};
\end{equation*}

\vspace{-0.3cm}

\item[(ii)] for all $k\in\zz$ and
$x,\,\wz x,\,y\in (0, \fz)$ with $|x-\wz x|\le (2^{-k}+|x-y|)/2$,
 \begin{eqnarray*}
 &&|S_k(x,\,y)-S_k(\wz x,\,y)|\\
 &&\quad\le
 \wz C\frac 1{m_\lz(I(x, 2^{-k})) +m_\lz(I(y, 2^{-k}))+
m_\lz(I(x, |x-y|))}\frac {|x-\wz x|^{\ez_1}2^{-k\ez_2}}{(2^{-k}+|x-y|)^{\ez_1+\ez_2}};
\end{eqnarray*}

\vspace{-0.3cm}

\item[(iii)] property (ii) also holds with $x$ and $y$ interchanged;

\vspace{-0.3cm}

\item[(iv)] for all $k\in\zz$ and $x,\,\wz x,\,y,\,\wz y\in (0, \fz)$
with $|x-\wz x|\le (2^{-k}+|x-y|)/3$ and $|y-\wz y|\le (2^{-k}+|x-y|)/3$,
\begin{equation*}
\begin{array}[b]{cl}
&\lf|[S_k(x,\,y)-S_k(x,\,\wz y)]-[S_k(\wz x,\,y)-S_k(\wz x,\,\wz y)]\r|\\
&\quad \le \wz C\dfrac 1{m_\lz(I(x, 2^{-k})) +m_\lz(I(y, 2^{-k}))+
m_\lz(I(x, |x-y|))}\frac {|x-\wz x|^{\ez_1}|y-\wz y|^{\ez_1}2^{-k\ez_3}}
{(2^{-k}+|x-y|)^{2\ez_1+\ez_3}};
\end{array}
\end{equation*}

\vspace{-0.3cm}

\item[(v)] for all $k\in\zz$ and $x\in (0, \fz)$,
$\int_0^\fz S_k(x,z)\,\dmz(z)=1=
\int_0^\fz S_k(z,x)\,\dmz(z).$
\end{enumerate}
\end{defn}

\begin{rem}\label{r1.1}\rm
Similarly to \cite[Remark 2.1(ii)]{yz}, if a
sequence $\{\wz S_t\}_{t>0}$ of bounded linear integral
operators on $\ltz$ satisfies (i) through (v) of Definition \ref{d1.1}
 with $2^{-k}$ replaced by $t$, then we call $\{\wz S_t\}_{t>0}$
 a {\it continuous approximation of the identity of $(\ez_1, \ez_2, \ez_3)$}
 (in short, continuous $(\ez_1, \ez_2, \ez_3)-\ati$).
 For example, if $\{S_k\}_{k\in\zz}$
 is an $(\ez_1, \ez_2, \ez_3)-\ati$ and if we set $\wz S_t(x, y)\equiv
 S_k(x, y)$ for $t\in(2^{-k-1}, 2^{-k}]$ with $k\in\zz$, then
 $\{\wz S_t\}_{t>0}$ is a continuous $(\ez_1, \ez_2, \ez_3)-\ati$.
\end{rem}

The following space of test functions was introduced in \cite{hmy}; see
also \cite{hmy06}.

\begin{defn}\rm\label{d1.2}
Let $x_1\in (0, \fz)$, $r\in (0, \fz)$, $\bz\in (0, 1]$
 and $\gz\in (0, \fz)$. A function $\phi$ on $(0, \fz)$
 is said to belong to the {\it space of test functions},
 $\cg(x_1, r, \bz, \gz)$, if
there exists a positive constant $C$ such that

${\rm (G)_i}\ |\phi(x)|\le C\frac1{m_\lz(I(x,\,r+|x-x_1|))}
\lf(\frac r{r+|x-x_1|}\r)^\gz$
 for all $x\in (0, \fz)$;

${\rm (G)_{ii}}\ |\phi(x)-\phi(y)|\le C
\lf(\frac{|x-y|}{r+|x_1-x|}\r)^\bz\frac1{m_\lz(I(x,\,r+|x-x_1|))}
\lf(\frac r{r+|x_1-x|}\r)^\gz$ for all  $x$, $y\in(0, \fz)$
satisfying that $|x-y|\le (r+|x_1-x|)/2$.

Moreover, for any $f\in \cg(x_1, r, \bz, \gz)$, its {\it norm} is defined by
\begin{equation*}
\|f\|_{\cg(x_1,\, r,\,\bz,\,\gz)}\equiv\inf\{C:\
{\rm (G)_i}\ {\rm and}\ {\rm (G)_{ii}}\ {\rm hold}\}.
\end{equation*}
\end{defn}

\begin{rem}\rm\label{r1.2}
(i) Let $\{S_t\}_{t>0}$ be an $(\ez_1, \ez_2, \ez_3)-\ati$ for
some positive constants $\ez_1$, $\ez_2$ and $\ez_3$, and
$S_t(x, y)$ be the kernel of $S_t$. Obviously, $S_t(x, \cdot)$ for
any fixed $t$ and $x\in(0, \fz)$
is a test function of type $(x, t, \ez_1, \ez_2)$, and $S_t(\cdot, y)$ for any
fixed $t$ and $y\in(0, \fz)$ is a test function of type $(y, t, \ez_1, \ez_2)$; see
also \cite[p.\,19]{hmy}.

(ii) For any $x\in(0, \fz)$, $1+|x-1|\sim 1+x$.
By this fact together with \eqref{1.4},
if we take $x_1\equiv1$ and $r\equiv1$ in Definition \ref{d1.2},
we have that if $\phi\in \cg(1, 1, \bz, \gz)$, then

${\rm (G)_{i'}}\ |\phi(x)|\le C\frac1{(1+x)^{2\lz+1+\gz}}$
 for all $x\in (0, \fz)$;

${\rm (G)_{ii'}}\ |\phi(x)-\phi(y)|\le C
\frac{|x-y|^\bz}{(1+x)^{2\lz+1+\gz+\bz}}$ for all  $x$, $y\in(0, \fz)$
satisfying that $|x-y|\le (1+|x-1|)/2$.
\end{rem}

The {\it space} $\ocg(x, r, \bz, \gz)$ is defined to be the set of all functions
$f\in \cg(x, r, \bz, \gz)$ such that $\int_0^\fz f(x)\,dm_\lz(x)=0$. Moreover,
we endow the space $\ocg(x, r, \bz, \gz)$ with the same norm as the space
$\cg(x, r, \bz, \gz)$.

The space $\cg(x, r, \bz, \gz)$ is a Banach space. Let
$\ez\in(0, 1]$ and $\bz,\,\gz\in (0, \ez]$. We
further define the {\it space} $\cg^\ez_0(x, r, \bz, \gz)$
to be the completion of the set $\cg(x, r, \ez, \ez)$ in $\cg(x, r, \bz, \gz)$.
For $f\in \cg^\ez_0(x, r, \bz, \gz)$, define
$\|f\|_{\cg^\ez_0(x,\, r,\, \bz,\, \gz)}\equiv\|f\|_{\cg(x,\, r,\, \bz,\, \gz)}$.
Let $(\cg^\ez_0(x, r, \bz, \gz))'$ be the {\it set of all continuous linear
functionals on $\cg^\ez_0(x, r, \bz, \gz)$}
endowed with the weak$^\ast$ topology,
and denote by $\langle f, \vz\rangle$ the natural pairing of elements
$f\in (\cg^\ez_0(x, r, \bz, \gz))'$ and $\vz\in \cg^\ez_0(x, r, \bz, \gz)$.
Throughout this paper, we fix $x_1\equiv1$ and write
$\cg^\ez_0(\bz, \gz)\equiv\cg^\ez_0(1, 1, \bz, \gz),$
and $(\cg^\ez_0(\bz, \gz))'\equiv(\cg^\ez_0(1, 1, \bz, \gz))'.$

Similarly, let the {\it space} $\ocg^\ez_0(x, r, \bz, \gz)$ be the
completion of the set $\ocg(x, r, \ez, \ez)$  in the space $\ocg(x, r, \bz, \gz)$.
For any $f\in \ocg^\ez_0(x, r, \bz, \gz)$, define
$\|f\|_{\ocg^\ez_0(x,\, r,\, \bz,\, \gz)}\equiv \|f\|_{\ocg(x,\, r,\, \bz,\, \gz)}$.
Denote by $(\ocg^\ez_0(x, r, \bz, \gz))'$ the {\it space of all continuous
linear functionals on $\ocg^\ez_0(x, r, \bz, \gz)$}, and
endow $(\ocg^\ez_0(x, r, \bz, \gz))'$ with the weak$^\ast$ topology. We
always write $\ocg^\ez_0(\bz, \gz)\equiv\ocg^\ez_0(1, 1, \bz, \gz)$.
See \cite{hmy} or \cite{hmy06} for the details. Moreover, it was proved
in \cite{yz09} that for any $\ez,\,\wz\ez\in(0, 1)$
and $\bz,\,\gz\in (0, \min\{\ez,\wz\ez\})$,
the spaces $\cg^\ez_0(\bz, \gz)=\cg^{\wz\ez}_0(\bz, \gz)$ and
$\ocg^\ez_0(\bz, \gz)=\ocg^{\wz\ez}_0(\bz, \gz).$

We now recall the atomic Hardy spaces $\hpz$ in \cite{hmy}; see also
\cite{cw77}.

\begin{defn}\label{d1.3}\rm
Let $p\in ((2\lz+1)/(2\lz+2), 1]$, $\ez\in (0, 1)$ and
$\bz$, $\gz\in (0, \ez)$. A function $a$ is called an
{\it $\hpz$-atom} if there exists an open bounded interval $I\subset (0,\fz)$
such that $\supp(a)\subset I$, $\|a\|_\ltz\le[m_\lz(I)]^{1/2-1/p}$
and $\int_0^\fz a(x)\,\dmz(x)=0$.

A distribution $f\in (\cg^\ez_0(\bz, \gz))'$ is said to belong to
the {\it Hardy space} $\hpz$ if
$f=\sum_{j=1}^\fz\az_j a_j$ in $(\cg^\ez_0(\bz, \gz))'$, where for every $j$,
$a_j$ is an $\hpz$-atom and $\az_j\in\cc$ satisfying that
$\sum_{j=1}^\fz|\az_j|^p<\fz$. The {\it norm} $\|f\|_\hpz$ of
$f$ in $\hpz$ is defined by
$\|f\|_\hpz\equiv\inf\{(\sum_{j=1}^\fz|\az_j|^p)^{1/p}\},$
where the infimum is taken over all the decompositions of $f$ as above.
\end{defn}

The following class $\zlz$ of functions is a slight variant
of the corresponding class introduced in \cite{bdt}.

\begin{defn}\label{d1.4}\rm
Let $\zlz$ be the set of all $C^2$-functions $\phi$ on
$[0, \fz)$ such that for all $x\ge0$,
\begin{equation}\label{1.6}
0\le \phi(x)\le C(1+x^2)^{-\lz-1},
\end{equation}
\begin{equation}\label{1.7}
\lf|\phi'(x)\r|\le Cx(1+x^2)^{-\lz-2},
\end{equation}
and
\begin{equation}\label{1.8}
\lf|\phi''(x)\r|\le C(1+x^2)^{-\lz-2}.
\end{equation}
\end{defn}

We now recall the radial maximal function, the nontangential maximal function,
the grand maximal function, the Littlewood-Paley $g$-function
and the Lusin-area function in \cite{hmy}.

\begin{defn}\label{d1.5}\rm
Let $\phi\in\zlz$, $\ez\in(0, 1]$ and $\bz$, $\gz\in (0, \ez)$.
For any $f\in (\cg^\ez_0(\bz, \gz))'$, the {\it radial maximal function},
the {\it nontangential maximal function}
and the {\it grand maximal function} are defined by
setting, for all $x\in(0, \fz)$,
\begin{equation*}
\cmr(f)(x)\equiv\sup_{t>0}\lf|\Phi_t(f)(x)\r|\equiv
\sup_{t>0}\lf|f\sharp_\lz\phi_t(x)\r|,
\end{equation*}
\begin{equation*}
\cmn(f)(x)\equiv\sup_{t>0,\,|x-y|<t}\lf|\Phi_t(f)(y)\r|\equiv\sup_{t>0,\,|x-y|<t}
\lf|f\sharp_\lz\phi_t(y)\r|
\end{equation*}
and
\begin{equation*}
\cmg(f)(x)\equiv\sup\{\langle f, \vz\rangle:\ \vz\in
\cg^\ez_0(\bz, \gz),\ \|\vz\|_{\cg(x,\,r,\,\bz,\,\gz)}\le1\ {\rm for\ some}\
r\in(0, \fz)\},
\end{equation*}
 where for all $t,\,y\in(0, \fz)$,
$\phi_t(y)\equiv t^{-2\lz-1}\phi(y/t)$.
\end{defn}

\begin{rem}\label{r1.3}\rm
Observe that the functions
\begin{equation*}
P^{[\lz]}(x)\equiv\frac{2\lz\bgz(\lz)}{\bgz(\lz+1/2)\sqrt{\pi}}
\frac1{(1+x^2)^{\lz+1}}
\end{equation*}
and $W^{[\lz]}(x)\equiv2^{(1-2\lz)/2}
\exp\lf(-x^2/2\r)/\bgz(\lz+1/2)$
both belong to $\zlz$. Recall that $\plz(f)\equiv e^{-t\sqrt\tbz}f=f\sharp_\lz
P^{[\lz]}_t$
and
$W^{[\lz]}_t(f)\equiv e^{-t\tbz}f=f\sharp_\lz W^{[\lz]}_{\sqrt{2t}}$
(see \cite[pp.\,200-201]{bdt}). Thus, the radial maximal functions
and the nontangential maximal functions respectively associated with
$\{e^{-t\sqrt{\tbz}}\}_{t>0}$ and $\{e^{-t\tbz}\}_{t>0}$ are
special cases of $\cmr$ and $\cmn$.
\end{rem}

\begin{defn}\label{d1.6}\rm
Let $\ez_1\in(0, 1]$, $\ez_2$, $\ez_3\in(0, \fz)$,
$a\in(0, \fz)$, $\ez\in (0, \min\{\ez_1, \ez_2\})$,
$\bz,\,\gz\in(0, \ez)$ and $\{S_k\}_{k\in\zz}$
be an $(\ez_1, \ez_2, \ez_3)-\mathop\mathrm{AOTI}$.
For all $f\in (\cg^\ez_0(\bz, \gz))'$, the {\it Littlewood-Paley $g$-function}
$g(f)$ and the {\it Lusin-area function} $S_a(f)$ are respectively defined by setting,
for all $x\in(0, \fz)$,
$$g(f)(x)\equiv\lf[\sum_{k\in\zz}|S_k(f)(x)-S_{k-1}(f)(x)|^2\r]^{1/2}$$
and
$$S_a(f)(x)\equiv\lf\{\sum_{k\in\zz}\dint_{|x-y|<a2^{-k}}
|S_k(f)(y)-S_{k-1}(f)(y)|^2\frac{\dmz(y)}{m_\lz(I(x, a2^{-k}))}\r\}^{1/2}.$$
\end{defn}

The first result of this paper is as follows.

\begin{thm}\label{t1.1}
Let $\lz\in (0, \fz)$, $p\in((2\lz+1)/(2\lz+2), 1]$,
$\ez_1\in(0, 1]$, $\ez_2$, $\ez_3\in(0, \fz)$,
$a\in(0, \fz)$, $\ez\in (0, \min\{\ez_1, \ez_2\})$,
$\bz,\,\gz\in ((2\lz+1)(1/p-1), \ez)$
and $\phi\in\zlz$.
A distribution $f\in (\cg^\ez_0(\bz, \gz))'$ belongs to $\hpz$ if and only if
$\cm(f)\in\lpz$; moreover,
$$\|f\|_\hpz\sim\|\cm(f)\|_\lpz,$$
where $\cm(f)$ is one of $\cmr(f)$, $\cmn(f)$, $\cmg(f)$, $g(f)$ and $S_a(f)$.
\end{thm}

Differently from \cite{bdt}, we establish Theorem \ref{t1.1}
by first showing that $\{\Phi_t\}_{t>0}$ defined as in Definition
\ref{d1.5} is actually a constant multiple of
an approximation of the identity
as in Definition \ref{d1.1}; see Lemma \ref{l2.1} below.
We then obtain all desired conclusions
of Theorem \ref{t1.1} by directly applying
results in \cite{hmy,gly1,yz}. The details
are given in Section \ref{s2}.

By applying Theorem \ref{t1.1}, we next establish the
characterization of $\hpz$ in terms of
the Riesz transform $\riz$. Let $r\in[1, \fz)$ and $f\in\lrz$.
The $\Delta_\lz$-conjugate of $f$ is defined by setting,
for any $t,\,x\in(0, \fz)$,
\begin{equation*}
\qlz(f)(x)\equiv\int_0^\fz\qlz(x, y)f(y)\, dm_\lz(y),
\end{equation*}
where for any $t,\,x,\,y\in(0, \fz)$,
\begin{equation}\label{1.9}
\begin{array}[b]{cl}
\qlz(x, y)&\equiv-(xy)^{-\lz+1/2}\dint_0^\fz e^{-t\xi}\xi J_{\lz+1/2}(x\xi)
J_{\lz-1/2}(\xi y)\,d\xi\\
&=-\dfrac{2\lz}{\pi}\dint_0^\pi\dfrac{(x-y\cos\theta)(\sin\theta)^{2\lz-1}}
{(x^2+y^2+t^2-2xy\cos\theta)^{\lz+1}}\,d\theta;
\end{array}
\end{equation}
see \cite[p.\,84]{ms}. Moreover,
there exists the boundary value function $\lim_{t\to0}\qlz(f)(x)$
for almost every $x\in (0, \fz)$ (see \cite[p.\,84]{ms}),
which is defined to be the {\it Riesz transform $\riz(f)$}.
Muckenhoupt and Stein \cite[p.\,87]{ms} also proved that $\riz$ is bounded
on $\lrz$ when $r\in(1, \fz)$. In \cite[pp.\,710-711]{bfbmt},
Betancor et al. further showed that if
$r\in[1, \fz)$ and $f\in\lrz$, then for almost every $x\in (0, \fz)$,
\begin{equation*}
\riz(f)(x)=\dlim_{\dz\to0}\dint_{0,\,|x-y|>\dz}^\fz
Q^{[\lz]}_0(x, y)f(y)\,dm_\lz(y),
\end{equation*}
where for any $x,\,y\in(0, \fz)$,
\begin{equation*}
Q_0^{[\lz]}(x, y)\equiv-\dfrac{2\lz}{\pi}\dint_0^\pi
\frac{(x-y\cos\theta)(\sin\theta)^{2\lz-1}}
{(x^2+y^2-2xy\cos\theta)^{\lz+1}}\,d\theta.
\end{equation*}
Moreover, Betancor, Fari\~{n}a  and Sanabria \cite{bfs} showed
that $\riz$ is a Calder\'on-Zygmund operator on the space
$((0, \fz), \rho, dm_\lz)$ of homogeneous type,
where $\riz$ is bounded from $\hoz$ to $\loz$.

Let $\ez\in (0, 1]$, $\bz,\,\gz,\,\wz\ez\in(0, \ez)$  and
$f\in (\cg^{\wz\ez}_0(\bz, \gz))'$.  $\riz(f)$ is said to
belong to $(\ocg(\bz, \gz))'$,
if there exists $F\in (\ocg^{\wz\ez}_0(\bz, \gz))'$ such that for
all $\psi\in \ocg^{\wz\ez}_0(\bz, \gz)$,
\begin{equation*}
\langle \riz (f), \psi\rangle\equiv
\dint_0^\fz f(x)\triz(\psi)(x)\,dm_\lz(x)
=\dint_0^\fz F(x)\psi(x)\,dm_\lz(x),
\end{equation*}
where $\triz$ is the adjoint operator of $\riz$; see also
\cite[Lemma 2.42]{bdt}. By Lemma \ref{l3.1} below,
we see that $\triz$ is bounded from $\ocg^{\wz\ez}_0(\bz, \gz)$
to $\cg^{\wz\ez}_0(\bz, \gz)$.

A distribution $f\in (\cg^{\wz \ez}_0(\bz, \gz))'$ is said to be {\it restricted at
infinity}, if for any
$\phi\in\zlz$ and $r>0$ large enough, $f\sharp_\lz\phi\in \lrz$.
By Theorem \ref{t1.1} and an
argument as in \cite[pp.\,100-101]{s93}, we see that for any
$f\in \hpz$ with $p\in((2\lz+1)/(2\lz+2), 1]$ and
$\phi\in\zlz$, $f\sharp_\lz\phi\in \lrz$ for all $r\in[p, \fz]$.
Moreover, we have the following characterization of $\hpz$
in terms of the Riesz transform $\riz$.

\begin{thm}\label{t1.2}
Let $\lz\in (0, \fz)$, $p\in ((2\lz+1)/(2\lz+2), 1]$, $\ez\in (0, 1]$,
$\bz,\,\gz\in((2\lz+1)(1/p-1), \ez)$,
$\wz\ez\in(0, \ez)$, $\phi\in\zlz$ and $f\in (\cg^{\wz\ez}_0(\bz, \gz))'$ be
restricted at infinity. Then $f\in \hpz$ if and only if
there exists a positive constant $C$ such that for all $\dz\in(0, \fz)$,
\begin{equation}\label{1.10}
\lf\|f\slz \phi_\dz\r\|_\lpz+\lf\|\riz\lf(f\slz \phi_\dz\r)\r\|_\lpz\le C.
\end{equation}
\end{thm}

To show the necessity of Theorem \ref{t1.2}, by using
the molecular characterization of the atomic Hardy space
$\hpz$ (see Theorems 2.1 and 2.2 in \cite{hyz}) and
the boundedness criterion on sublinear operators on
Hardy spaces over RD-spaces (see \cite{yz08} or \cite{gly1}),
we first show that the Riesz transform $\riz$
is bounded from $\hpz$ to $\lpz$ (see Lemma \ref{l3.2}
below) and  $f\slz \phi_\dz$ is
bounded on $\hpz$ uniformly on $\dz\in (0,\fz)$, which
further induces the necessity of Theorem \ref{t1.2}.

To show the sufficiency of Theorem \ref{t1.2}, we establish a key estimate
for the radial maximal function of the first
entry of the conjugate harmonic systems satisfying
the generalized Cauchy-Riemann equations
associated with $\Delta_\lz$ in terms of maximal
$L^p$ norm of the conjugate harmonic systems (see Lemma \ref{l3.4}
below), which when $p=1$ was already obtained
by Betancor et al. in \cite[Lemma 2.38]{bdt}.
Then, as an application of Theorem \ref{t1.1},
we obtain the sufficiency of Theorem \ref{t1.2}.
The details are given in Section \ref{s3}.

We remark that \eqref{1.10} is formally slightly different from the case
for $H^p(\rn)$ (see \cite[p.\,123]{s93}).
Recall that a tempered
distribution $f\in{\mathcal S}'(\rn)$ restricted at infinity
belongs to $H^p(\rn)$ with $p\in ((n-1)/n, 1]$
if and only if there exists a positive constant $C$
such that for all $\dz\in(0, \fz)$,
\begin{equation}\label{1.11}
\lf\|f\ast\phi_\dz\r\|_\lp+\sum_{j=1}^n\lf\|R_j(f)\ast\phi_\dz\r\|_\lp\le C,
\end{equation}
where $\phi\in{\mathcal S}(\rn)$,
$\phi_\dz(x)\equiv \dz^{-n}\phi(x/\dz)$, and $\{R_j\}_{j=1}^n$ are
the classical Riesz transforms; see \cite[p.\,123]{s93}. Since
$\{R_j\}_{j=1}^n$ are convolution operators, we have that
for all $j\in\{1,\,\cdots, n\}$,
$$R_j(f)\ast \phi_\dz=(K_j\ast f)\ast\phi_\dz=K_j\ast (f\ast\phi_\dz)
= R_j(f\ast\phi_\dz),$$
where $K_j$ is the kernel of $R_j$. Thus, for $H^p(\rn)$
with $p\in ((n-1)/n, 1]$, \eqref{1.11} and \eqref{1.10}
are the same, and actually these commutative relations
were used in the proof of \cite[p.\,123, Proposition 3]{s93}.
However, it is unclear if this is also true
for the Hardy space $\hpz$. Nevertheless,
from Theorems \ref{t1.1} and \ref{t1.2}
together with the boundedness of $\riz$ on $\hpz$ (see Lemma \ref{l3.2} below),
we immediately deduce the following result.
The details are omitted.

\begin{cor}\label{c1.1}
Let $\lz\in (0, \fz)$, $p\in ((2\lz+1)/(2\lz+2), 1]$, $\ez\in (0, 1]$,
$\bz,\,\gz\in((2\lz+1)(1/p-1), \ez)$,
$\wz\ez\in(0, \ez)$ and $f\in (\cg^{\wz\ez}_0(\bz, \gz))'$ be
restricted at infinity. Then $f\in \hpz$ if and only if
there exists a positive constant $\wz C$ such that for all $\dz\in(0, \fz)$,
\begin{equation*}
\lf\|f\slz \phi_\dz\r\|_\lpz+\lf\|\riz\lf(f\slz \phi_\dz\r)\r\|_\lpz
+\lf\|\lf(\riz(f)\r)\slz \phi_\dz\r\|_\lpz\le \wz C.
\end{equation*}
\end{cor}

Finally, we make some conventions on notation. Throughout the paper,
we denote by $C$ and $\wz C$ {\it positive constants} which
are independent of the main parameters, but they may vary from line to
line. If $f\le Cg$, we then write $f\ls g$ or $g\gs f$;
and if $f \ls g\ls f$, we then write $f\sim g.$
For any $k\in (0, \fz)$ and $I\equiv I(x, r)$ for some $x$, $r\in (0, \fz)$,
$kI\equiv(x-kr, x+kr)\cap(0, \fz)$.

\section{\hspace{-0.6cm}.\hspace{0.4cm}Proof of Theorem \ref{t1.1}}\label{s2}

\noindent This section is devoted to the proof of Theorem \ref{t1.1}.
We start with the following key lemma.

\begin{lem}\label{l2.1}
Let $\phi\in \zlz$ and $\{\Phi_t\}_{t>0}$ be as in Definition \ref{d1.5}.
Then
$$\lf\{\frac 1{\|\phi\|_\loz}\Phi_t\r\}_{t>0}$$
is a continuous $(1, 1, 1)-\mathop\mathrm{AOTI}$.
\end{lem}

\begin{proof}\rm
By Definition \ref{d1.5} and \eqref{1.1}, we see that for all
$t,\,x,\,y\in(0, \fz)$, the kernel
$\Phi_t(x, y)\equiv \tlz\phi_t(y)$.
For all $x,\,y,\,z\in(0, \fz)$, let
$\triangle(x, y, z)$ be the area of a triangle with sides $x$, $y$,
$z$ when such a triangle exists, and
\begin{equation*}
D(x, y, z)\equiv\frac{2^{2\lz-2}\bgz(\lz+\frac12)}
{\bgz(\lz)\sqrt\pi}(xyz)^{-2\lz+1}[\triangle(x, y, z)]^{2\lz-2}
\end{equation*}
when such $\triangle(x, y, z)$ exists, and zero otherwise.
Then by \eqref{1.2} and the change of variables, we obtain that
\begin{equation}\label{2.1}
\begin{array}[b]{cl}
\Phi_t(x, y)&=\dfrac{\bgz(\lz+1/2)}{\bgz(\lz)\sqrt{\pi}}
\dint_0^\pi \phi_t\lf(\sqrt{x^2+y^2-2xy\cos \theta}\r)
(\sin\theta)^{2\lz-1}\,d\theta\\
&=\dfrac{\bgz(\lz+1/2)}{\bgz(\lz)\sqrt{\pi}}
\dint_{|x-y|}^{x+y}\phi_t(z) (xy)^{-1}z
\lf[1-\lf(\frac{x^2+y^2-z^2}{2xy}\r)^2\r]^{\lz-1}\,dz\\
&=\dint_0^\fz \phi_t(z)D(x, y, z)\,\dmz(z).
\end{array}
\end{equation}
From this together with (6) in \cite[p.\,335]{h}
and the change of variables, we further deduce
that for all $x,\,t\in (0, \fz)$,
\begin{equation}\label{2.2}
\dint_0^\fz \Phi_t(x, y)\,\dmz(y)
=\dint_0^\fz\dint_0^\fz \phi_t(z)D(x, y, z)\,\dmz(y)\,\dmz(z)
=\|\phi\|_\loz.
\end{equation}

By the homogeneity of $\loz$, we may assume that $\|\phi\|_\loz=1$.
Then by \eqref{1.2} and
\eqref{2.2}, $\Phi_t(x, y)$ is symmetric in $x$ and $y$, and satisfies
(v) of Definition \ref{d1.1}.
Moreover, it follows from \eqref{2.2} that $\{\Phi_t\}_{t>0}$ is
uniformly bounded on both $\loz$ and
$\linz$, which together with the Marcinkiewicz interpolation theorem
yields that $\{\Phi_t\}_{t>0}$ is also uniformly
bounded on $\ltz$. This can also be deduced from \eqref{2.2}, the H\"older inequality,
the symmetry of $\Phi_t(x, y)$ and the Fubini theorem as follows:
for all $f\in \ltz$,
\begin{equation*}
\begin{array}[t]{cl}
\dint_0^\fz\lf|\dint_0^\fz\Phi_t(x, y)f(y)\,dm_\lz(y)\r|^2\,dm_\lz(x)&
\le \dint_0^\fz \dint_0^\fz \Phi_t(x, y)|f(y)|^2\,dm_\lz(y)\,dm_\lz(x)\\
&=\dint_0^\fz|f(y)|^2\,dm_\lz(y).
\end{array}
\end{equation*}

Thus, by the symmetry of $\Phi_t(x, y)$,
to finish the proof of Lemma \ref{l2.1}, we still need to show that
$\{\Phi_t\}_{t>0}$ satisfies (i), (ii) and (iv) of Definition \ref{d1.1}.
We first prove that $\{\Phi_t\}_{t>0}$ satisfies Definition \ref{d1.1}(i).
By \eqref{1.4}, we obtain that for all $x,\,y,\,t\in (0, \fz)$,
\begin{equation}\label{2.3}
m_\lz(I(y, t))\ls m_\lz(I(x, t))+m_\lz(I(x, |x-y|)).
\end{equation}
Then Definition \ref{d1.1}(i) is reduced to showing that
for all $x,\,y,\,t\in (0, \fz)$,
\begin{equation}\label{2.4}
\Phi_t(x, y)\ls \frac 1{m_\lz(I(x, t)) +
m_\lz(I(x, |x-y|))}\frac {t}{t+|x-y|}.
\end{equation}
To this end, by \eqref{1.2} and \eqref{1.6}, we see that
\begin{equation}\label{2.5}
\begin{array}[b]{cl}
\Phi_t(x, y)&\sim t^{-2\lz-1}\dint_0^\pi\phi
\lf(\frac{\sqrt{x^2+y^2-2xy\cos \theta}}{t}\r)(\sin\theta)^{2\lz-1}\,d\theta\\
&\ls t^{-2\lz-1}\dint_0^\pi
\lf(1+\frac{x^2+y^2-2xy\cos \theta}{t^2}\r)^{-\lz-1}
(\sin\theta)^{2\lz-1}\,d\theta\\
&\ls \dint_0^\pi\frac{t(\sin\theta)^{2\lz-1}}{(t^2+x^2+y^2-2xy\cos \theta)^{\lz+1}}
\,d\theta.
\end{array}
\end{equation}

We then consider the following two cases.

{\it Case (i)} $t\ge x$ or $|x-y|\ge x/2$. In this case, from
\eqref{2.5} and the fact that
\begin{equation}\label{2.6}
\dint_0^\pi (\sin\theta)^{2\lz-1}\,d\theta=
\dfrac{\Gamma(\lz)\sqrt\pi}{\Gamma(\lz+\frac12)},
\end{equation}
we deduce that
\begin{equation*}
\Phi_t(x, y)\ls \frac t{(t^2+|x-y|^2)^{\lz+1}},
\end{equation*}
which together with \eqref{1.4} yields \eqref{2.4}.

{\it Case (ii)} $t< x$ and $|x-y|< x/2$. In this case, \eqref{2.4}
follows from
\begin{equation}\label{2.7}
\Phi_t(x, y)\ls \frac t{x^{2\lz}(t+|x-y|)^2}.
\end{equation}
Observe that in this case, $x\sim y$. Then by the fact that for all
$\theta\in(0, \pi/2]$, $\sin\theta\sim\theta$ and
$1-\cos\theta \ge 2(\theta/\pi)^2$, we have
\begin{equation*}
\begin{array}[t]{cl}
{\rm E}_1&\equiv\dint_0^{\pi/2}
\frac{t(\sin\theta)^{2\lz-1}}
{[t^2+|x-y|^2+2xy(1-\cos\theta)]^{\lz+1}}\,d\theta\\
&\ls\dint_0^{\pi/2}\frac{t\theta^{2\lz-1}}
{[t^2+|x-y|^2+4xy\theta^2/{\pi^2}]^{\lz+1}}\,d\theta\\
&\ls \dfrac t{(xy)^\lz(t^2+|x-y|^2)}\dint_0^\fz
\dfrac{\bz^{2\lz-1}}{(1+\bz^2)^{\lz+1}}\,d\bz\ls
\dfrac t{x^{2\lz}(t^2+|x-y|^2)}.
\end{array}
\end{equation*}
On the other hand, by \eqref{2.6} and the fact that
$\cos\theta<0$ for all $\theta\in (\pi/2, \pi]$, we have that
\begin{equation*}
\begin{array}[t]{cl}
{\rm E}_2&\equiv\dint_{\pi/2}^\pi
\frac{t(\sin\theta)^{2\lz-1}}
{(t^2+x^2+y^2-2xy\cos \theta)^{\lz+1}}\,d\theta\\
&\ls \dfrac t{x^{2\lz}(t^2+|x-y|^2)}\dint_{\pi/2}^\pi
(\sin\theta)^{2\lz-1}\,d\theta\ls\frac t{x^{2\lz}(t^2+|x-y|^2)},
\end{array}
\end{equation*}
which together with the estimate of ${\rm E}_1$ yields \eqref{2.7}.

We now show that $\{\Phi_t\}_{t>0}$ satisfies Definition \ref{d1.1}(ii). By
\eqref{2.3}, it suffices to show that for all
$x,\,\wz x,\,y,\,t\in (0, \fz)$ with $|x-\wz x|\le (t+|x-y|)/2$,
\begin{equation}\label{2.8}
{\rm F}\equiv|\Phi_t(x, y)-\Phi_t(\wz x, y)|\ls
\frac 1{m_\lz(I(x, t)) +
m_\lz(I(x, |x-y|))}\frac {t|x-\wz x|}{(t+|x-y|)^2}.
\end{equation}
Using the mean value theorem and \eqref{1.7}, we obtain that
\begin{equation}\label{2.9}
\begin{array}[b]{cl}
{\rm F}&\ls\dint_0^\pi t^{-2\lz-1}\lf|\phi
\lf(\frac{\sqrt{x^2+y^2-2xy\cos \theta}}{t}\r)\r.\\
&\quad\lf.-\phi\lf(\dfrac{\sqrt{\wz x^2+y^2-2\wz xy\cos \theta}}{t}\r)\r|
(\sin\theta)^{2\lz-1}\,d\theta\\
&\ls\dint_0^\pi t^{-2\lz-2}\lf|\phi'
\lf(\frac{\sqrt{\xi^2+y^2-2\xi y\cos\theta}}{t}\r)\r||x-\wz x|
(\sin\theta)^{2\lz-1}\,d\theta\\
&\ls\dint_0^{\pi}\frac{t|x-\wz x|(\sin\theta)^{2\lz-1}}
{(t^2+\xi^2+y^2-2\xi y\cos\theta)^{\lz+\frac32}}\,d\theta,
\end{array}
\end{equation}
where $\az\in(0, 1)$ and $\xi\equiv(1-\az)x+\az \wz x$.

We now prove \eqref{2.8} by considering the following two cases.

{\it Case (i)} $t\ge x$ or $|x-y|\ge x/2$. In this case,
\eqref{2.8} follows from
\begin{equation}\label{2.10}
{\rm F}\ls \frac {t|x-\wz x|}{(t+|x-y|)^{2\lz+3}}.
\end{equation}
By $|x-\wz x|\le (t+|x-y|)/2$ and the choice of $\xi$, we have
\begin{equation}\label{2.11}
t+|\xi-y|=t+|(1-\az)x+\az \wz x-y|\ge t+|x-y|-\az|x-\wz x|> (t+|x-y|)/2.
\end{equation}
Then \eqref{2.10} follows from \eqref{2.11} together with
\eqref{2.9} and \eqref{2.6} easily.

{\it Case (ii)} $t< x$ and $|x-y|< x/2$. In this case,
$$|x-\xi|<|x-\wz x|\le (t+|x-y|)/2<3x/4$$
 and hence $y\sim x\sim \wz x\sim \xi$.
This fact together with the fact that for all
$\theta\in (0, \pi/2]$, $\sin\theta\sim \theta$ and
$1-\cos\theta\ge 2(\theta/\pi)^2$ further implies that
\begin{equation*}
\begin{array}[t]{cl}
{\rm F}_1&\equiv\dint_0^{\pi/2}
\frac{t|x-\wz x|(\sin\theta)^{2\lz-1}}
{[t^2+|\xi-y|^2+2\xi y(1-\cos\theta)]^{\lz+\frac32}}\,d\theta\\
&\ls\dint_0^{\pi/2}\frac{t|x-\wz x|\theta^{2\lz-1}}
{[t^2+|\xi-y|^2+4\xi y\theta^2/{\pi^2}]^{\lz+\frac32}}\,d\theta\\
&\ls \dfrac {t|x-\wz x|}{(\xi y)^\lz(t+|x-y|)^3}\dint_0^\fz
\frac{\bz^{2\lz-1}}{(1+\bz^2)^{\lz+\frac32}}\,d\bz\ls
\frac {t|x-\wz x|}{x^{2\lz}(t+|x-y|)^3}.
\end{array}
\end{equation*}
On the other hand, from \eqref{2.11} and \eqref{2.6},
we deduce that
\begin{equation*}
\begin{array}[t]{cl}
{\rm F}_2&\equiv\dint_{\pi/2}^\pi
\frac{t|x-\wz x|(\sin\theta)^{2\lz-1}}
{(t^2+\xi^2+y^2-2\xi y\cos\theta)^{\lz+\frac32}}\,d\theta\\
&\ls \dfrac {t|x-\wz x|}{x^{2\lz}(t+|x-y|)^3}\dint_{\pi/2}^\pi
(\sin\theta)^{2\lz-1}\,d\theta\ls\frac {t|x-\wz x|}{x^{2\lz}(t+|x-y|)^3}.
\end{array}
\end{equation*}
Combining the estimates of ${\rm F}_1$
and ${\rm F}_2$, we obtain ${\rm F}\ls \frac {t|x-\wz x|}{x^{2\lz}(t+|x-y|)^3},$
which implies \eqref{2.8}.

Similarly, to show that $\{\Phi_t\}_{t>0}$ satisfies Definition \ref{d1.1}(iv), by
\eqref{2.3}, it suffices to show that for all $x$, $\wz x$, $y$, $\wz y$,
$t\in (0, \fz)$
with $|x-\wz x|\le (t+|x-y|)/3$ and $|y-\wz y|\le (t+|x-y|)/3$,
\begin{equation}\label{2.12}
\begin{array}[b]{cl}
{\rm G}&\equiv\lf|[\Phi_t(x,\,y)-\Phi_t(x,\,\wz y)]-
[\Phi_t(\wz x,\,y)-\Phi_t(\wz x,\,\wz y)]\r|\\
&\quad \ls\dfrac 1{m_\lz(I(x, t))+
m_\lz(I(x, |x-y|))}\frac {t|x-\wz x||y-\wz y|}{(t+|x-y|)^3}.
\end{array}
\end{equation}
Using the mean value theorem, \eqref{1.7} and \eqref{1.8}, we obtain that
\begin{equation*}
\begin{array}[t]{cl}
{\rm G}&\ls
\dint_0^\pi t^{-2\lz-1}|x-\wz x||y-\wz y|
\lf[\lf|\frac1 {t^2}\phi''\lf(
\dfrac{\sqrt{\xi_1^2+\xi_2^2-2\xi_1\xi_2\cos\theta}}t\r)\r|\r.\\
&\quad\lf.+\dfrac1{t\sqrt{\xi_1^2+\xi_2^2-2\xi_1\xi_2\cos\theta}}
\lf|\phi'\lf(\dfrac{\sqrt{\xi_1^2+\xi_2^2-2\xi_1\xi_2\cos\theta}}{t}
\r)\r|\r](\sin\theta)^{2\lz-1}\,d\theta\\
&\ls \dint_0^{\pi}\frac{t|x-\wz x||y-\wz y|(\sin\theta)^{2\lz-1}}
{(t^2+\xi_1^2+\xi_2^2-2\xi_1\xi_2\cos\theta)^{\lz+2}}\,d\theta,
\end{array}
\end{equation*}
where $\az$, $\bz\in(0, 1)$, $\xi_1\equiv(1-\az)x+\az \wz x$
and $\xi_2\equiv(1-\bz)y+\bz \wz y$.

To prove \eqref{2.12}, we consider the following two cases.

{\it Case (i)} $t\ge x$ or $|x-y|\ge x/2$ or $t\ge y$. In this case,
by $|x-\wz x|\le (t+|x-y|)/3$ and $|y-\wz y|\le (t+|x-y|)/3$, we have
 \begin{equation}\label{2.13}
t+|\xi_2-\xi_1|\ge t+|y-x|-|x-\wz x|-|y-\wz y|\ge (t+|x-y|)/3.
 \end{equation}
 This together with \eqref{2.6} yields that
 \begin{equation*}
{\rm G}\ls \frac{t|x-\wz x||y-\wz y|}{(t+|x-y|)^{2\lz+4}},
 \end{equation*}
 which implies \eqref{2.12} in this case.

 {\it Case (ii)} $t<x$, $|x-y|< x/2$ and $t< y$.
 In this case, $\wz x$, $y\in(x/2, 3x/2)$ and $\wz y\in(x/6, 13x/6)$.
 Moreover, we see that $\xi_2\sim \wz y\sim y\sim x\sim \wz x\sim \xi_1$.
 From this together with \eqref{2.13}, \eqref{2.6}
 and the fact that for all $\theta\in(0, \pi/2]$,
 $\sin\theta\sim \theta$ and $1-\cos\theta\ge 2(\theta/\pi)^2$,
 we further deduce that
 \begin{equation*}
\begin{array}[t]{cl}
{\rm G}&\ls \dint_0^{\pi/2}
\frac{t|x-\wz x||y-\wz y|(\sin\theta)^{2\lz-1}}
{[t^2+|\xi_1-\xi_2|^2+2\xi_1\xi_2(1-\cos\theta)]^{\lz+2}}\,d\theta
+\dint_{\pi/2}^\pi\cdots\\
&\ls \dfrac {t|x-\wz x||y-\wz y|}{(\xi_1\xi_2)^\lz(t+|\xi_1-\xi_2|)^4}\dint_0^\fz
\dfrac{\bz^{2\lz-1}}{(1+\bz^2)^{\lz+2}}\,d\bz+
\dfrac{t|x-\wz x||y-\wz y|}{(t^2+\xi_1^2+\xi_2^2)^{\lz+2}}\\
&\ls\dfrac{t|x-\wz x||y-\wz y|}{x^{2\lz}(t+|x-y|)^4}.
\end{array}
 \end{equation*}
This implies \eqref{2.12} and hence finishes the proof of Lemma \ref{l2.1}.
\end{proof}

\begin{proof}[Proof of Theorem \ref{t1.1}]\rm
Let $p$, $\ez$, $\bz$, $\gz$ and $a$ be as in Theorem 1.1,
$\phi\in\zlz$ and $f\in (\cg^\ez_0(\bz, \gz))'$.
From \cite[Theorem 3.1]{yz} (see also \cite[Corollary 1.8]{gly2})
together with Lemma \ref{l2.1}, we deduce that
$f\in\hpz$ if and only if
$\cmr(f)\in\lpz$ and $\|f\|_\hpz\sim\|\cmr(f)\|_\lpz$.
Furthermore, \cite[Corollary 4.18]{gly1} implies that
$f\in\hpz$ if and only if $\cmn(f)\in\lpz$ if and only if $\cmg(f)\in\lpz$); moreover,
$$\|f\|_\hpz\sim\|\cmn(f)\|_\lpz\sim\lf\|\cmg(f)\r\|_\lpz.$$
Finally, it follows from \cite[Theorems 5.13 and 5.16]{hmy} that
$f\in\hpz$ if and only if $g(f)\in\lpz$ if and only if $S_a(f)\in\lpz$; moreover,
$$\|f\|_\hpz\sim\|g(f)\|_\lpz\sim\lf\|S_a(f)\r\|_\lpz.$$
Combining these facts, we then complete the proof of Theorem \ref{t1.1}.
\end{proof}

\section{\hspace{-0.6cm}.\hspace{0.4cm}Proof of Theorem \ref{t1.2}}\label{s3}

\noindent
In this section, we present the proof of Theorem \ref{t1.2}.
We begin with the following lemma on the boundedness of
$\triz$.

\begin{lem}\label{l3.1}
Let $\ez\in (0, 1]$ and $\bz,\,\gz,\,\wz\ez\in(0, \ez)$. Then
$\triz$ is bounded from $\ocg^{\wz\ez}_0(\bz, \gz)$ to $\cg^{\wz\ez}_0(\bz, \gz)$.
\end{lem}

\begin{proof}
By \cite[p.\,87]{ms}, we have that $\riz$ is bounded on
$\lrz$ for $r\in (1, \fz)$,  and  so is $\triz$.
Recall that the kernel, denoted by $\triz(x, y)$,  of $\triz$ satisfies that
for all $x,\, y\in (0, \fz)$,
\begin{equation*}
\triz(x, y)= Q^{[\lz]}_0(y, x)\equiv-\frac{2\lz}{\pi}
\dint_0^\pi\frac{(y-x\cos\theta)(\sin\theta)^{2\lz-1}}
{(x^2+y^2-2xy\cos\theta)^{\lz+1}}\,d\theta;
\end{equation*}
see Lemma 2.42 in \cite{bdt}.
It is easy to see that $\triz(x, y)$
satisfies Conditions (i)-(iii) of Theorem 2.18 in
\cite{hmy} with $\ez=1$ therein (see also \cite{bfs}), which means the  $\triz(x, y)$
also satisfies the hypotheses of Corollary 2.24 in \cite{hmy}.
By this and \cite[Remark 2.14(iii), Remark 2.17, Corollary 2.24]{hmy},
Lemma \ref{l3.1} is
reduced to showing that $\triz(1)\in \bmoz$ is a constant, where
$f\in L^1_{\rm loc}(0, \fz)$ is called to belong to
the {\it space} $\bmoz$ if
\begin{equation*}
\sup_{x,\,r\in(0,\,\fz)}\frac1{m_\lz(I(x,\,r))}\dint_{I(x,\,r)}
\lf|f(y)-\frac1{m_\lz(I(x, r))}\dint_{I(x,\,r)}f(z)\,dm_\lz(z)\r|\,dm_\lz(y)<\fz.
\end{equation*}
Recall that $\bmoz$ is the dual space of $\hoz$ and
$\riz$ is bounded from $\hoz$ to $\loz$ (see \cite{cw77}).
Then by Theorem 4.10 in \cite{r}, we see that $\triz$ is bounded
from $\linz$ to $\bmoz$, which implies that $\triz(1)\in \bmoz$.
Moreover, for any $x\in(0, \fz)$,
\begin{equation*}
\begin{array}[t]{cl}
\triz(1)(x)&=\dlim_{\dz\to0}\dint_{0,\,|x-y|>\dz}^\fz Q_0^{[\lz]}(y, x)y^{2\lz}\,dy\\
&=-\dfrac{2\lz}{\pi}\dlim_{\dz\to0}\dint_{0,\,|x-y|>\dz}^\fz
\dint_0^\pi\frac{(y-x\cos\theta)(\sin\theta)^{2\lz-1}}
{(x^2+y^2-2xy\cos\theta)^{\lz+1}}\,d\theta y^{2\lz}\,dy\\
&=-\dfrac{2\lz}{\pi}\dlim_{\dz\to0}\int_{0,\,|1-z|>\dz}^\fz
\dint_0^\pi\frac{(z-\cos\theta)(\sin\theta)^{2\lz-1}}
{(z^2+1-2z\cos\theta)^{\lz+1}}\,d\theta z^{2\lz}\,dz=\triz(1)(1).
\end{array}
\end{equation*}
This implies that $\triz(1)$ is a constant and hence finishes the proof
Lemma \ref{l3.1}.
\end{proof}

We next establish the boundedness of $\riz$ on $\hpz$. To this end,
for any $x_0,\,x,\,r\in (0, \fz)$, let $d_\lz(x, x_0)\equiv|\int_x^{x_0}y^{2\lz}\,dy|$,
$I_{d_\lz}(x_0, r)\equiv \{x\in(0, \fz): d_\lz(x, x_0)<r\}$
and for all $k\in\nn$,
$$R_k(I_{d_\lz}(x_0, r))\equiv\{x\in (0, \fz):\,\, 2^{k-1}
m_\lz(I_{d_\lz}(x_0, r))\le d_\lz(x, x_0)<
2^km_\lz(I_{d_\lz}(x_0, r))\}.$$
Then $d_\lz$ is the measure distance
and satisfies that for any $x_0,\,r\in(0, \fz)$,
\begin{equation}\label{3.1}
m_\lz(I_{d_\lz}(x_0, r))\sim r;
\end{equation}
see Theorem 3 in \cite{ms1}.
{\it Write $H^p((0, \fz), dm_\lz)$  as
$H^p((0, \fz), \rho, dm_\lz)$ for the moment},
where $\rho(x,y)\equiv|x-y|$ for all $x,\,y\in(0, \fz)$.
If we replace $\rho$ by $d_\lz$ in Definition \ref{d1.3},
we then obtain the corresponding $H^p((0, \fz), d_\lz, dm_\lz)$-atoms and
the Hardy spaces $H^p((0, \fz), d_\lz, dm_\lz)$. In \cite[Theorems 2.1]{hyz},
it was proved that for any $p\in((2\lz+1)/(2\lz+2), 1]$, the spaces
\begin{equation}\label{3.2}
H^p((0, \fz), \rho, dm_\lz)= H^p((0, \fz), d_\lz, dm_\lz)
\end{equation}
with equivalent norms.

We recall the notion of
molecules in \cite{hyz} as follows; see also \cite{cw77, ms2}.

\begin{defn}\label{d3.1}
Let $p\in((2\lz+1)/(2\lz+2), 1]$ and $\eta\equiv\{\eta_k\}_{k\in\nn}\subset[0, \fz)$
such that $\sum_{k=1}^\fz k\eta_k<\fz$ when $p=1$, or
$\sum_{k=1}^\fz(\eta_k)^p 2^{k(1-p)}<\fz$ when $p\in(0, 1)$.
A function $M\in\ltz$ is called a $(p, 2, \eta)$-molecule centered at an interval
$I_{d_\lz}\equiv I_{d_\lz}(y_0, r_0)$ for some $y_0$, $r_0\in(0, \fz)$, if

${\rm (M)_i}$ $\|M\|_\ltz\le [m_\lz(I_{d_\lz})]^{1/2-1/p};$

${\rm (M)_{ii}}$ for all $k\in\nn$,
$\|M\chi_{R_k(I_{d_\lz})}\|_\ltz\le \eta_k 2^{-k/2}[m_\lz(I_{d_\lz})]^{1/2-1/p};$

${\rm (M)_{iii}}$ $\int_0^\fz M(x)x^{2\lz}\,dx=0.$
\end{defn}

\begin{lem}\label{l3.2}
Let $p\in ((2\lz+1)/(2\lz+2), 1]$. Then $\riz$ is bounded from $\hpz$ to
$\lpz$ and bounded on $\hpz$.
\end{lem}

\begin{proof}
We only show that $\riz$ is bounded on $\hpz$, since the proof for
the boundedness of $\riz$ from $\hpz$ to $\lpz$ is similar and easier.
Assume that $a$ is an $\hpz$-atom such that
$\supp(a)\subset I\equiv I(x_0, r)$ for some $x_0,\,r\in(0, \fz)$.
By Theorem 1.1 in \cite{yz08} (see also \cite[Theorem 5.9]{gly1}),
we only need to show that there exists a positive constant $C$,
independent of $a$, such that $\|\riz(a)\|_\hpz\le C$.
From Theorem 2.2 in \cite{hyz}, we deduce that there exists a
positive constant $\wz C$ such that for any $(p, 2, \eta)$-molecule
$M$ as in  Definition \ref{d3.1},
$M\in H^p((0, \fz), d_\lz, dm_\lz)$ and $\|M\|_{H^p((0, \,\fz),
\,d_\lz,\,dm_\lz)}\le \wz C$.
Via this and \eqref{3.2}, it suffices to show that
$\riz(a)$ is a $(p, 2, \eta)$-molecule
centered at the interval
$I_{d_\lz}\equiv I_{d_\lz}(x_0, m_\lz(I))$
with $\eta\equiv\{2^{-\frac k{2\lz+1}}\}_{k\in\nn}$.

Recall that $\riz$ is bounded from $\hoz$ to $\loz$. This together with $a\in \hoz$
implies that $\riz(a)\in\loz$. Then from this
observation and Lemma \ref{l3.1}, it follows that
\begin{equation*}
\dint_0^\fz \riz(a)(x)x^{2\lz}\,dx=\langle\riz(a), 1\rangle
=\langle a, \triz(1)\rangle=0.
\end{equation*}
On the other hand, by \eqref{3.1}, we see that
\begin{equation}\label{3.3}
m_\lz(I_{d_\lz})\sim m_\lz(I)
\end{equation}
and for each $k\in\nn$,
\begin{equation}\label{3.4}
m_\lz(R_k(I_{d_\lz}))\ls 2^km_\lz(I).
\end{equation}
Applying the boundedness of $\riz$ on $\ltz$, Definition \ref{d1.3} and \eqref{3.3},
we have
\begin{equation}\label{3.5}
\begin{array}[t]{cl}
\lf\|\riz (a)\r\|_\ltz&\ls\|a\|_\ltz
\ls [m_\lz(I)]^{1/2-1/p}
\ls [m_\lz(I_{d_\lz})]^{1/2-1/p}.
\end{array}
\end{equation}
Thus, via \eqref{3.3}, Lemma \ref{l3.2} is reduced to showing that for each $k\in\nn$,
\begin{equation}\label{3.6}
\begin{array}[t]{cl}
\lf\|\lf[\riz (a)\r]\chi_{R_k\lf(I_{d_\lz}\r)}\r\|_\ltz
\ls (2^k)^{-\frac {\lz+3/2}{2\lz+1}}[m_\lz(I)]^{1/2-1/p}.
\end{array}
\end{equation}

If $x\in 2I$, then by \eqref{1.5}, we see that
\begin{equation*}
d_\lz(x, x_0)=\lf|\int_x^{x_0}y^{2\lz}\,dy\r|
\le m_\lz(2I)\ls 2^{2\lz+1}m_\lz(I).
\end{equation*}
This together with \eqref{3.1}  and \eqref{3.3} implies that
there exists $K_0\in\nn$ such that $(2I)\cap R_k(I_{d_\lz})=\emptyset$
for all $k> K_0$.
Furthermore, \eqref{3.5} implies \eqref{3.6} for all $k\in\nn$ and $k\le K_0$.

To prove \eqref{3.6} for $k> K_0$,
we first claim that for any $x\in (0, \fz)\setminus (2I)$,
\begin{equation}\label{3.7}
\lf|\riz(a)(x)\r|\ls \frac{r[m_\lz(I)]^{1-1/p}}{|x-x_0|^{2\lz+2}},
\end{equation}
and when $x_0\ge 2r$,
\begin{equation}\label{3.8}
\lf|\riz(a)(x)\r|\ls \frac{r[m_\lz(I)]^{1-1/p}}{|x-x_0|^2x^\lz x_0^\lz}.
\end{equation}
Indeed, by the vanishing moment of $a$ and the mean value theorem, we obtain that
\begin{equation}\label{3.9}
\begin{array}[b]{cl}
\lf|\riz(a)(x)\r|&\sim
\lf|\dint_I\lf[Q^{[\lz]}_0(x, y)-Q^{[\lz]}_0(x, x_0)\r]a(y)y^{2\lz}\,dy\r|\\
&\ls r\dint_I\dint_0^\pi\frac{(\sin\theta)^{2\lz-1}}
{(x^2+\xi^2-2x\xi\cos\theta)^{\lz+1}}\,d\theta|a(y)|y^{2\lz}\,dy,
\end{array}
\end{equation}
where $\xi\equiv\az x_0+(1-\az)y$ for $y\in I$ and some $\az\in (0, 1)$.

On the one hand, since $x\in (0, \fz)\setminus (2I)$, then
 $x^2+\xi^2-2x\xi\cos\theta\ge(x-\xi)^2\gs(x-x_0)^2$,
 which together with \eqref{3.9} implies \eqref{3.7}.
On the other hand, when $x_0\ge 2r$, we have that for all
$y\in I=(x_0-r, x_0+r)$, $y\sim x_0$. This implies that $\xi\sim x_0$.
Then by some estimates similar to those used in the
estimate \eqref{2.7}, we further see that
\begin{equation*}
\dint_0^\pi\frac{(\sin\theta)^{2\lz-1}}
{(x^2+\xi^2-2x\xi\cos\theta)^{\lz+1}}\,d\theta\ls
\frac1{|x-x_0|^2x^\lz x_0^\lz}.
\end{equation*}
This together with \eqref{3.9} and $\|a\|_\loz\le [m_\lz(I)]^{1-1/p}$
implies \eqref{3.8}.

For $x> 2x_0$, the mean value theorem implies that
$$d_\lz(x, x_0)\sim x^{2\lz+1}-x_0^{2\lz+1}\sim \zeta^{2\lz}(x-x_0)\ls
x^{2\lz}(x-x_0)\ls(x-x_0)^{2\lz+1}$$
for some $\beta\in (0, 1)$ and  $\zeta\equiv\beta x_0+(1-\bz)x$.
This together with $x\in R_k(I_{d_\lz})$ leads to that
$2^km_\lz(I)\ls (x-x_0)^{2\lz+1}$.
By this, \eqref{3.7}, \eqref{3.4} and \eqref{1.4}, we further see that
\begin{equation*}
\begin{array}[t]{cl}
{\rm F_1}&\equiv\lf\{\dint_{2x_0}^\fz\lf[\riz (a)(x)\r]^2
\chi_{R_k\lf(I_{d_\lz}\r)}(x)\,dm_\lz(x)\r\}^{1/2}\\
&\ls r[m_\lz(I)]^{1-1/p}\lf\{\dint_{R_k\lf(I_{d_\lz}\r)}
\frac1{|x-x_0|^{4\lz+4}}\,dm_\lz(x)\r\}^{1/2}\\
&\ls r[m_\lz(I)]^{1-1/p}[2^km_\lz(I)]^{-\frac {\lz+3/2}{2\lz+1}}
\ls [m_\lz(I)]^{1/2-1/p}(2^k)^{-\frac {\lz+3/2}{2\lz+1}}.
\end{array}
\end{equation*}
Similarly,
\begin{equation*}
\begin{array}[t]{cl}
{\rm F_2}&\equiv\lf\{\dint_{0}^{x_0/2}\lf[\riz (a)(x)\r]^2
\chi_{R_k\lf(I_{d_\lz}\r)}(x)\,dm_\lz(x)\r\}^{1/2}
\ls [m_\lz(I)]^{1/2-1/p}(2^k)^{-\frac {\lz+3/2}{2\lz+1}}.
\end{array}
\end{equation*}

Assume that there exists $x\in [x_0/2, 2x_0]\cap R_k(I_{d_\lz})$. Then
since $R_k(I_{d_\lz})\cap (2I)=\emptyset$, we see that $[x_0/2, 2x_0]\nsubseteq (2I)$,
which further implies that
$x_0> 2r$. Using this fact together with the mean value theorem and the fact that
$x\sim x_0$, we have that $d_\lz(x, x_0)\sim x_0^{2\lz}|x-x_0|$.
This implies that for any $x\in R_k(I_{d_\lz})$,
$|x-x_0|\sim 2^km_\lz(I)x_0^{-2\lz}.$
Thus, from this fact together with \eqref{3.8} and \eqref{1.4}, we deduce that
\begin{equation*}
\begin{array}[t]{cl}
{\rm F}_3&\equiv\lf\{\dint_{x_0/2}^{2x_0}\lf[\riz (a)(x)\r]^2
\chi_{R_k\lf(I_{d_\lz}\r)}(x)\,dm_\lz(x)\r\}^{1/2}\\
&\ls r[m_\lz(I)]^{1-1/p}\lf\{\dint_{R_k\lf(I_{d_\lz}\r)}
\frac1{|x-x_0|^4x_0^{4\lz}}\,dm_\lz(x)\r\}^{1/2}\\
&\ls r[m_\lz(I)]^{1-1/p}\lf\{\dfrac{m_\lz(R_k(I_{d_\lz}))}{[2^km_\lz(I)]^4
x_0^{-4\lz}}\r\}^{1/2}\ls r[m_\lz(I)]^{1-1/p}[2^km_\lz(I)]^{-3/2}x_0^{2\lz}\\
&\ls [m_\lz(I)]^{1/2-1/p}(2^k)^{-\frac 32}
\ls [m_\lz(I)]^{1/2-1/p}(2^k)^{-\frac {\lz+3/2}{2\lz+1}}.
\end{array}
\end{equation*}

Combining the estimates of ${\rm F}_i$ for $i\in\{1,\,2,\,3\}$,  we obtain
\eqref{3.6}, which completes the proof of Lemma \ref{l3.2}.
\end{proof}

Recall that for any $t,\,x,\,y\in(0, \fz)$,
the kernel $\plz(x, y)$ of $\plz$ satisfies that
\begin{equation}\label{3.10}
\begin{array}[b]{cl}
\plz(x, y)&=\dint_0^\fz e^{-tz}(xz)^{-\lz+1/2}J_{\lz-1/2}(xz)(yz)^{-\lz+1/2}
J_{\lz-1/2}(yz)\, dm_\lz(z)\\
&=\dfrac{2\lz t}{\pi}\dint_0^\pi\dfrac{(\sin\theta)^{2\lz-1}}
{(x^2+y^2+t^2-2xy\cos\theta)^{\lz+1}}\,d\theta;
\end{array}
\end{equation}
see \cite{bdt} or \cite{w}.
Moreover, $\{\plz\}_{t>0}$ is a contraction semigroup on $\lrz$
for all $r\in[1, \fz]$. Also, for any $f\in \lrz$ with $r\in[1, \fz]$,
the Poisson integral $u(t, x)\equiv \plz (f)(x)$ satisfies the differential equation
that for all $t,\,x\in (0, \fz)$,
\begin{equation}\label{3.11}
\pa^2_t u(t, x)+\pa^2_x u(t, x)+\frac{2\lz}{x}\pa_x u(t, x)=0.
\end{equation}
Moreover, let $v(t, x)\equiv \qlz(f)(x)$. Then $u$ and  $v$ satisfy
the following Cauchy-Riemann type equations:
\begin{equation}\label{3.12}
\pa_t v+\pa_x u=0,\quad \pa_tv-\pa_xu-\frac{2\lz}{x}v=0;
\end{equation}
see \cite{ms, bdt}.

The following lemma was proved in \cite{ms}.

\begin{lem}\label{l3.3}
Suppose that

(i) $u(t, x)$ is continuous in $t\in[0, \fz)$, $x\in\rr$ and even in $x$;

(ii) In the region where $u(t, x)>0$, $u$ is of class $C^2$ and satisfies
$\partial^2u/\partial t^2+\partial^2u/\partial x^2
+2\lz x^{-1}\partial u/\partial x\ge0$ there;

(iii) $u(0, x)=0;$

(iv) For some $r\in[1, \fz)$, there exists a positive constant $\wz C$ such that
\begin{equation}\label{3.13}
\sup_{0<t<\fz}\dint_0^\fz|u(t, x)|^r\,dm_\lz(x)\le \wz C<\fz.
\end{equation}
Then $u(t, x)\le0$.

\end{lem}

For any function $u$ on $(0, \fz)\times \rr$ and
$x\in (0, \fz)$, let $u^\ast(x)\equiv\sup_{t>0}|u(t, x)|$.
We then have a variant of Lemma 2.38 in \cite{bdt} as follows.

\begin{lem}\label{l3.4}
Let $p\in ((2\lz)/(2\lz+1), 1]$.
Assume that $u(t, x)$ and $v(t, x)$ are, respectively,
even and odd (with respect to $x$)
real valued $C^2$ functions on $(0, \fz)\times \rr$ and satisfy
the following Cauchy-Riemann type equation \eqref{3.12}.
Let $F\equiv(u^2+v^2)^{1/2}$ and suppose that
\begin{equation}\label{3.14}
\sup_{t>0}\dint_0^\fz [F(t, x)]^p\,dm_\lz(x)<\fz.
\end{equation}
Then $u^*\in \lpz$ and there exists a positive constant
$C$, depending only on $\lz$ and $p$, such that
\begin{equation*}
\|u^*\|_\lpz\le C\sup_{t>0}\dint_0^\fz [F(t, x)]^p\,dm_\lz(x).
\end{equation*}
\end{lem}

\begin{proof}
By \cite[Lemma 5]{ms}, for any $p\ge (2\lz)/(2\lz+1)$, we have
that for all $(t, x)\in (0, \fz)\times \rr$ such that $F(t, x)>0$,
\begin{equation}\label{3.15}
\frac{\pa^2 [F(t, x)]^p}{\pa x^2}+\frac{\pa^2 [F(t, x)]^p}{\pa t^2}
+\frac{2\lz}{x}\frac{\pa [F(t, x)]^p}{\pa x}\ge0.
\end{equation}
Let $p\in ((2\lz)/(2\lz+1), 1]$ and $F_\dz(x)\equiv F(\dz, x)$
for any $(\dz, x)\in (0, \fz)\times\rr$. Take $q\in ((2\lz)/(2\lz+1), p)$
and $r\equiv p/q$. An application of \eqref{3.14} leads to that
for all $\dz\in (0, \fz)$, $F^q_\dz\in\lrz$.

We claim that for any $t,\,\dz,\,x\in(0, \fz)$,
\begin{equation}\label{3.16}
[F(t+\dz, x)]^q\le \plz(F^q_\dz)(x).
\end{equation}
To see this, let
$P_0^{[\lz]}(F_\dz^q)\equiv \lim_{t\to0_+}\plz(F_\dz^q)$
and define $V_\dz(t, x)$ by setting, for all $t\in [0, \fz)$ and
$x\in \rr$,
$$V_\dz(t, x)\equiv [F(t+\dz, x)]^q-\wz P^{[\lz]}_t(F^q_\dz)(x),$$
where $\wz P^{[\lz]}_t(F^q_\dz)$ is the odd extension to
$\rr$ of $\plz(F^q_\dz)$. To show \eqref{3.16},
it suffices to show that Lemma \ref{l3.3} holds for $V_\dz$.
In fact, it is obvious that $V_\dz$ is continuous on $[0, \fz)\times \rr$.
Since $F$ is nonnegative, $\plz(F_\dz^q)$ is also nonnegative.
Thus, if $(t, x)\in (0, \fz)$ such that
$V_\dz(t, x)>0$, then $[F(t+\dz, x)]^q>0$, which together with
\eqref{3.11} and \eqref{3.15} implies that
\begin{equation*}
\frac{\pa^2 V_\dz(t, x)}{\pa x^2}+\frac{\pa^2 V_\dz(t, x)}{\pa t^2}
+\frac{2\lz}{x}\frac{\pa V_\dz(t, x)}{\pa x}\ge0
\end{equation*}
and hence Lemma \ref{l3.3}(ii). Moreover,
by the continuity of $F$ and the fact that
$\plz f\to f$ in $\lpz$ for any $f\in \lpz$ with $p\in [1, \fz$),
we see that $P_0^{[\lz]}(F_\dz^q)=F_\dz^q$,
which yields that for any $x\in \rr$, $V_\dz(0, x)=0$. Thus,
Lemma \ref{l3.3}(iii) holds.

Finally, by the uniform boundedness of $\plz$ on $\lrz$ for $r\in [1, \fz]$ and
\eqref{3.14}, we have
\begin{equation*}
\begin{array}[t]{cl}
&\dsup_{0<t<\fz}\dint_0^\fz|V_\dz(t, x)|^r\,dm_\lz(x)\\
&\quad\ls\dsup_{0<t<\fz}\dint_0^\fz[F(t, x)]^p\,dm_\lz(x)
+\dsup_{0<t<\fz}\dint_0^\fz\lf[\plz(F_\dz^q)(x)\r]^r\,dm_\lz(x)\\
&\quad\ls\dsup_{t>0}\dint_0^\fz [F(t, x)]^p\,dm_\lz(x)<\fz.
\end{array}
\end{equation*}
Therefore, \eqref{3.13} holds for $V_\dz$ and consequently, the
claim follows from Lemma \ref{l3.3}.

Since $\{F_\dz^q\}_{\dz>0}$ is bounded on $\lrz$, and
$\lrz$ is reflexive, there exists a sequence $\dz_k\downarrow0$ and
 $h\in \lrz$ such that $F_{\dz_k}^q$ converges weakly to $h$
 in $\lrz$ as $k\to\fz$. Moreover, by the H\"older inequality, we see that
\begin{equation}\label{3.17}
\begin{array}[b]{cl}
\|h\|_\lrz^r&=
\lf\{\dsup_{\|g\|_\lrpz\le1}\lf|\dint_0^\fz g(x)h(x)\,dm_\lz(x)\r|\r\}^r\\
&=\lf\{\dsup_{\|g\|_\lrpz\le1}\lim_{k\to\fz}
\lf|\dint_0^\fz g(x)[F_{\dz_k}(x)]^q\,dm_\lz(x)\r|\r\}^r\\\
&\le \dlimsup_{k\to\fz}\lf\|F_{\dz_k}^q\r\|_\lrz^r\le
\dsup_{t>0}\dint_0^\fz [F(t, x)]^p\,dm_\lz(x).
\end{array}
\end{equation}

Since $F$ is continuous, then for any $x\in (0, \fz)$,
$[F(t+\dz_k, x)]^q\to [F(t, x)]^q$ as $k\to\fz$.
Observe that for any fixed $x\in (0, \fz)$, $\plz(x, \cdot)\in\lrpz$.
From this fact, we deduce that for each $x\in (0, \fz)$,
$\plz(F^q_{\dz_k})(x)\to \plz(h)(x)$
as $k\to\fz$.
Thus, by these facts and \eqref{3.16}, we have that
for any $t,\,x\in (0, \fz)$,
\begin{equation*}
[F(t, x)]^q=\dlim_{k\to\fz}[F(t+\dz_k, x)]^q\le
\dlim_{k\to\fz}\plz(F^q_{\dz_k})(x)=\plz(h)(x).
\end{equation*}
Therefore,
\begin{equation*}
[u^\ast(x)]^q=\lf[\dsup_{0<t<\fz} u(t, x)\r]^q
\le\dsup_{0<t<\fz}\lf[F(t, x)\r]^q\le\dsup_{0<t<\fz}\plz(h)(x).
\end{equation*}
By this together with (c) in \cite[p.\,86]{ms} and
\eqref{3.17}, we then have
\begin{equation*}
\begin{array}[t]{cl}
\|u^\ast\|_\lpz^p&=\|[u^\ast]^q\|_\lrz^r\ls\lf\|\dsup_{0<t<\fz}\plz(h)\r\|^r_\lrz\\
&\ls \|h\|^r_\lrz\ls\dsup_{t>0}\dint_0^\fz [F(t, x)]^p\,dm_\lz(x),
\end{array}
\end{equation*}
which completes the proof of Lemma \ref{l3.4}.
\end{proof}

\begin{proof}[Proof of Theorem \ref{t1.2}]
We first assume that $f\in \hpz$.
By Theorem \ref{t1.1}, we have that
$\sup_{\dz>0}|f\slz \phi_\dz|\in \lpz$ and  for all $\dz\in(0, \fz)$,
$$\lf\|f\slz \phi_\dz\r\|_\lpz\ls \|f\|_\hpz.$$
Thus, \eqref{1.10} is reduced to showing that for all $\dz\in(0, \fz)$,
\begin{equation}\label{3.18}
\lf\|\riz\lf(f\slz \phi_\dz\r)\r\|_\lpz\ls \|f\|_\hpz.
\end{equation}
To this end, for each $\dz\in(0, \fz)$, let
$\Phi_\dz(f)\equiv f\slz \phi_\dz$.
By an argument similar to that used in the estimates of
\eqref{3.5} and \eqref{3.6},
we see that for any $\hpz$-atom $a$, $\Phi_\dz(a)$ satisfies that
for some interval $I_{d_\lz}$,
\begin{equation*}
\begin{array}[t]{cl}
\lf\|\Phi_\dz(a)\r\|_\ltz \ls [m_\lz(I_{d_\lz})]^{1/2-1/p},
\end{array}
\end{equation*}
and for each $k\in\nn$,
\begin{equation*}
\begin{array}[t]{cl}
\lf\|\lf[\Phi_\dz(a)\r]\chi_{R_k\lf(I_{d_\lz}\r)}\r\|_\ltz
\ls (2^k)^{-\frac {\lz+3/2}{2\lz+1}}[m_\lz(I_{d_\lz})]^{1/2-1/p}.
\end{array}
\end{equation*}
Observe that \eqref{2.2} and the Fubini theorem imply that
$\int_0^\fz\Phi_\dz(a)(x)\,dm_\lz(x)=0$. These facts further
yield that there exists a positive constant $C$,
independent of $a$ and $\dz$, such that
$\Phi_\dz(a)/C$ is a $(p, 2, \eta)$-molecule
with $\eta\equiv\{2^{-\frac k{2\lz+1}}\}_{k\in\nn}$
as in Definition \ref{d3.1}.
Moreover, by this together with \cite[Theorem 2.2]{hyz}, we
see that $\Phi_\dz(a)\in  H^p((0, \fz), d_\lz, dm_\lz)$ and
$\|\Phi_\dz(a)\|_{ H^p((0,\, \fz),\, d_\lz,\, dm_\lz)}\ls 1$.
This combined with \cite[Theorem 2.1]{hyz} further implies that
$\Phi_\dz(a)\in  H^p((0, \fz), dm_\lz)$ and
$\|\Phi_\dz(a)\|_{H^p((0,\, \fz), \,dm_\lz)}\ls 1.$
From this and \cite[Theorem 1.1]{yz08}, we deduce that
$\{\Phi_\dz\}_{\dz>0}$ is bounded on $\hpz$
uniformly on $\dz\in (0,\fz)$.
On the other hand, by Lemma \ref{l3.2}, we obtain that $\riz$ is
bounded from $\hpz$ to $\lpz$ for all
$p\in((2\lz+1)/(2\lz+2), 1]$. Combining these two facts, we obtain \eqref{3.18}.

We now assume that \eqref{1.10} holds. For $\dz,\,t,\, x\in (0, \fz)$,
let $u(t, x)\equiv\plz(f\slz \phi_\dz)(x)$
and $v(t, x)\equiv\qlz(f\slz \phi_\dz)(x)$.
Moreover, define $F_\dz\equiv(\wz u^2+\wz v^2)^{1/2}$,
where $\wz u$ and $\wz v$ are even extension and odd
extension of $u$ and $v$ with
respect to $x$  to $\rr$, respectively.
Then by \eqref{3.10}, \eqref{1.9} and \eqref{3.12}
together with the definitions of $\wz u$ and
$\wz v$, we have that $\wz u,\,\wz v$ are $C^2$ functions on $(0, \fz)\times \rr$
satisfying the Cauchy-Riemann type equations \eqref{3.12}.

We prove that for all $x\in (0, \fz)$,
\begin{equation}\label{3.19}
[F_\dz(t,x)]^p\le P^{[\lz]}_t([F_\dz(0, \cdot)]^p)(x).
\end{equation}
To see this, for all $(t, x)\in[0, \fz)\times\rr$, define
$$V_\dz(t, x)\equiv[F_\dz(t,x)]^p- \wz P^{[\lz]}_t([F_\dz(0, \cdot)]^p)(x),$$
where $\wz P^{[\lz]}_t([F_\dz(0, \cdot)]^p)$ is the even extension to
$\rr$ of $P^{[\lz]}_t([F_\dz(0, \cdot)]^p)$.
As in the proof of Lemma \ref{l3.4}, \eqref{3.19}
is reduced to showing that $V_\dz$ satisfies Conditions (i)-(iv)
of Lemma \ref{l3.3}. Observe that $V_\dz$ is continuous on $[0, \fz)\times \rr$.
Moreover, we see that
$V_\dz$ satisfies (ii) and (iii) of Lemma \ref{l3.3}.
It remains to show that $V_\dz$ satisfies \eqref{3.13}.
Since the assumption that
$f$ is restricted at infinity implies that $f\slz \phi_\dz\in \lrz$
for all $\dz\in (0, \fz)$ and $r\in[p, \fz]$,
by the uniform boundedness of $\plz$ and $\qlz$ on $\ltz$ (see \cite[p.\,87]{ms}),
we further obtain that
\begin{equation}\label{3.20}
\begin{array}[b]{cl}
\dint_0^\fz [F_\dz(t, x)]^2\,dm_\lz(x)&=
\dint_0^\fz \lf[\lf|\plz(f\slz \phi_\dz)(x)\r|^2
+\lf|\qlz(f\slz \phi_\dz)(x)\r|^2\r]\,dm_\lz(x)\\
&\ls\dint_0^\fz\lf[f\slz \phi_\dz(x)\r]^2\,dm_\lz(x)<\fz.
\end{array}
\end{equation}

Observe that for almost every $x\in (0, \fz)$,
$$[F_\dz(0, x)]^2=|f\slz \phi_\dz(x)|^2+|\riz(f\slz \phi_\dz)(x)|^2.$$
This fact together with the boundedness of $\plz$ on $L^{2/p}((0, \fz), dm_\lz)$ and
 the boundedness of $\riz$ on $\ltz$ implies that
\begin{equation*}
\begin{array}[t]{cl}
\dint_0^\fz\lf[\plz(F_\dz^p(0, \cdot))(x)\r]^{2/p}\,dm_\lz(x)
&\ls\dint_0^\fz|F_\dz(0, x)|^2\,dm_\lz(x)\\
&\ls \dint_0^\fz\lf[\lf|f\slz \phi_\dz(x)\r|^2+
\lf|\riz\lf(f\slz \phi_\dz\r)(x)\r|^2\r]\,dm_\lz(x)\\
&\ls\dint_0^\fz\lf|f\slz \phi_\dz(x)\r|^2\,dm_\lz(x),
\end{array}
\end{equation*}
from which and \eqref{3.20}, we further deduce that
\begin{equation*}
\begin{array}[t]{cl}
&\dsup_{0<t<\fz}\dint_0^\fz|V_\dz(t, x)|^{2/p}\,dm_\lz(x)\\
&\quad\ls \dsup_{0<t<\fz}\dint_0^\fz \lf\{[F_\dz(t, x)]^2
+\lf[\plz(F_\dz^p(0, \cdot))(x)\r]^{2/p}\r\}\,dm_\lz(x)\\
&\quad\ls\dint_0^\fz\lf|f\slz \phi_\dz(x)\r|^2\,dm_\lz(x)<\fz.
\end{array}
\end{equation*}
Therefore, $V_\dz$ satisfies \eqref{3.13} and hence \eqref{3.19}
follows from Lemma \ref{l3.3} immediately.

Using \eqref{3.19}, \eqref{1.10} and the uniform boundedness of
$\{\plz\}_{t>0}$ on $\loz$,
we see that
\begin{equation}\label{3.21}
\begin{array}[b]{cl}
\dint_0^\fz[F_\dz(t, x)]^p\,dm_\lz(x)
&\le \dint_0^\fz P^{[\lz]}_t([F_\dz(0, \cdot)]^p)(x)\,dm_\lz(x)\\
&\ls\dint_0^\fz|F_\dz(0, x)|^p\,dm_\lz(x)\\
&\sim\dint_0^\fz\lf[\lf|f\slz \phi_\dz(x)\r|
+\lf|\riz\lf(f\slz \phi_\dz\r)(x)\r|\r]^p\,dm_\lz(x)\ls1.
\end{array}
\end{equation}

We claim that for each $t$, $x\in(0, \fz)$,
$F_\dz(t, x)\to F(t, x)$ as $\dz\to0$, where
$$F(t, x)\equiv \|\vz\|_\loz\{[\plz(f)(x)]^2+[\qlz(f)(x)]^2\}^{1/2}.$$
Indeed, observe that for any fixed $x\in (0, \fz)$,
$\plz(x, \cdot)$, $\qlz(x, \cdot)\in\cg(1,1)$. Thus
we only need to show that for all $\vz\in \zlz$,
$f\slz\vz_\dz\to f\|\vz\|_\loz$ in $(\cg(1,1))'$
as $\dz\to0$. To this end, let
$\psi\in \cg(1,1)$. Then by \eqref{2.1}, we have
that when $\dz\to0$,
\begin{eqnarray*}
&&\dint_0^\fz f\slz\vz_\dz(x)\psi(x)\,dm_\lz(x)\\
&&\quad=\dint_0^\fz\dint_0^\fz\dint_0^\fz
D(x, y, z)\vz_\dz(z)\,dm_\lz(z) f(y)\,dm_\lz(y)\psi(x)\,dm_\lz(x)\\
&&\quad=\dint_0^\fz\dint_0^\fz\dint_0^\fz D(x, y, z)\psi(x)\,dm_\lz(x)
\vz_\dz(z)\,dm_\lz(z) f(y)\,dm_\lz(y)\\
&&\quad=C_\lz\dint_0^\fz\dint_0^\fz\dint_0^\pi
\dz^{-2\lz-1}\psi(\sqrt{y^2+z^2-2yz\cos\theta})\\
&&\quad\quad\times\vz\lf(\frac{z}{\dz}\r)
(\sin\theta)^{2\lz-1}\,d\theta \,dm_\lz(z)\, f(y)\,dm_\lz(y)\\
&&\quad=C_\lz\dint_0^\fz\dint_0^\fz\dint_0^\pi
\psi(\sqrt{y^2+\dz^2z^2-2\dz yz\cos\theta})\vz(z)
(\sin\theta)^{2\lz-1}\,d\theta \,dm_\lz(z)\, f(y)\,dm_\lz(y)\\
&&\quad\to C_\lz\dint_0^\fz\dint_0^\fz\dint_0^\pi
\psi(y)\vz(z)(\sin\theta)^{2\lz-1}\,d\theta \,dm_\lz(z)\, f(y)\,dm_\lz(y)\\
&&\quad=\|\vz\|_\loz\dint_0^\fz\psi(y)f(y)\,dm_\lz(y),
\end{eqnarray*}
where $C_\lz\equiv\frac{\bgz(\lz+1/2)}{\bgz(\lz)\sqrt{\pi}}$.
This implies the claim.

By this claim together with \eqref{3.21} and
the Fatou lemma, we further have that
\begin{equation*}
\dsup_{0<t<\fz}\dint_0^\fz|F(t, x)|^p\,dm_\lz(x)\ls1,
\end{equation*}
from which together with  Lemma \ref{l3.4}, we deduce that
\begin{equation*}
\dint_0^\fz\dsup_{0<t<\fz}\lf|\plz(f)(x)\r|^p\,dm_\lz(x)\ls1.
\end{equation*}
Therefore, by Theorem \ref{t1.1}, we have $f\in \hpz$
and hence complete the proof of Theorem \ref{t1.2}.
\end{proof}

\medskip

\noindent{\bf Acknowledgement}

\medskip

\noindent Dachun Yang would like to thank
Professor Jorge Juan Betancor for some helpful discussions
on the subject of this paper. The first author is supported by the National
Natural Science Foundation (Grant No. 10871025) of China.

%
%
%
%
%
%
%
%
%
%
%

\end{document}